\documentclass{birkjour_t2}
 \newtheorem{thm}{Theorem}[section]
 
 \newtheorem{lem}[thm]{Lemma}
 
 \theoremstyle{definition}
 \newtheorem{defn}[thm]{Definition}
 \theoremstyle{remark}
 \newtheorem{rem}[thm]{Remark}
 
 \numberwithin{equation}{section}

\begin{document}

\title[Eigenvalues of the negative $(p,q)$-Laplacian under a Steklov-like BC]{Eigenvalues of the negative $(p,q)$-Laplacian under a Steklov-like boundary condition}

\author[Lumini\c{t}a Barbu]{Lumini\c{t}a Barbu}
\address{ Ovidius University, Faculty of Mathematics and Computer Science,\\ 124 Mamaia Blvd,
 900527 Constan\c{t}a, Romania}
  \email{lbarbu@univ-ovidius.ro}

\author[Gheorghe Moro\c{s}anu]{Gheorghe Moro\c{s}anu}
\address{  Central European University, Department of Mathematics,\\ Nador u. 9, 1051 Budapest, Hungary }

\email {morosanug@ceu.edu}

\begin{abstract}
In this paper we consider in a bounded domain $\Omega \subset \mathbb{R}^N$ with smooth boundary an eigenvalue problem for the negative $(p,q)$-Laplacian 
with a Steklov type boundary condition, where $p\in (1,\infty)$, $q\in (2,\infty)$ and $p\neq q$. A full description of the set of eigenvalues of 
this problem is provided, thus essentially extending a recent result by Abreu and Madeira [1] related to the $(p,2)$-Laplacian.  
\end{abstract}

\subjclass{35J60, 35J92, 35P30.}
\keywords{eigenvalues, $(p,q)$-Laplacian, Steklov-like boundary condition, Sobolev space, Nehari manifold, variational methods.}
\maketitle

\section{ Introduction}
In this paper we investigate the eigenvalue problem
\begin{equation}\label{eq:1.1}
\left\{\begin{array}{l}
A u:=-\Delta_p u-\Delta_q u=\lambda a(x) \mid u\mid ^{q-2}u\ \ \mbox{in} ~ \Omega,\\[1mm]
\frac{\partial u}{\partial\nu_A}=\lambda b(x) \mid u\mid ^{q-2}u \hspace*{2.6cm}~ \mbox{on} ~ \partial\Omega,
\end{array}\right.
\tag{1.1}
\end{equation}
under the following hypotheses

\bigskip

$(H_{pq}) \ \ \ \ \ \ p\in (1, \infty),~  q\in (2, \infty),~ p\neq q$;

\bigskip

$(H_{\Omega}) \ \ \ \ \ \ \Omega\subset \mathbb{R}^N, ~N\geq 2$, is a bounded domain with smooth boundary $\partial \Omega$;

\bigskip

$(H_{ab}) \ \ \ \ \ \ a,b\in L^{\infty}(\Omega)$ are given nonnegative functions satisfying
\begin{equation}\label{eq:1.2}
 \int_\Omega a(x)~dx+\int_{\partial\Omega} b(x)~d\sigma >0.
 \tag{1.2}
\end{equation}
We have denoted
$$
\frac{\partial u}{\partial\nu_A}:=\big(\mid \nabla u\mid ^{p-2}+\mid \nabla u\mid ^{q-2}\big)\frac{\partial u}{\partial\nu},
$$
where $\nu$ is the unit outward normal to $\partial\Omega$. As usual $\Delta_p$ denotes the $p$-Laplacian, i.e.,
$\Delta_pu=div \, (|\nabla u|^{p-2}\nabla u)$. The operator $\Delta_p + \Delta q$,
called $(p,q)$-Laplacian, occurs in quantum field theory.

\bigskip

The solution $u$ of \eqref{eq:1.1} is understood in a weak sense, as an element of the Sobolev space $W:=W^{1,\max\{p,q\}}(\Omega)$ satisfying
equation $\eqref{eq:1.1}_1$ in the sense of distributions and $\eqref{eq:1.1}_2$ in the sense of traces. Using a Green type formula
(see \cite{CF}, p. 71) we can define the concept of an eigenvalue of our problem as follows:
\begin{defn}\label{def1}
$\lambda\in \mathbb{R}$ is an eigenvalue of problem \eqref{eq:1.1} if there exists $u_\lambda\in W \setminus \{0\}$ such that
\begin{equation}\label{eq:1.3}
\begin{split}
\int_\Omega \Big(\mid \nabla u_\lambda\mid ^{p-2}+&\mid \nabla u_\lambda\mid ^{q-2}\Big)\nabla u_\lambda \cdot \nabla v~dx \\
&=\lambda\Big(\int_\Omega a\mid  u_\lambda\mid ^{q-2} u_\lambda  v~dx+\int_{\partial\Omega} b \mid  u_\lambda\mid ^{q-2} u_\lambda  v~d\sigma\Big)~\forall~v\in W.
\end{split}
\tag{1.3}
\end{equation}
\end{defn}
Indeed, according to the mentioned Green type formula, $u\in W$ is a solution of \eqref{eq:1.1} if and only if it satisfies \eqref{eq:1.3}.

\bigskip

Our goal is to determine the set of all eigenvalues of problem \eqref{eq:1.1}. Fortunately we are able to offer a complete description of this set (see
Theorem 3.1 below). It is worth pointing out that this nice result is due to the fact that operator $A$ is nonhomogeneous ($p\neq q$). The homogeneous case
($p=q$) is more delicate. For example, if $p=q$, $a\equiv 1$ and $b\equiv 0$, then the eigenvalue set of the corresponding (Neumann type) problem is
fully known only if $p=q=2$; otherwise, i.e. if $p=q\in (1, \infty)\setminus \{ 2\}$, then it is only known that, as a consequence of the
Ljusternik-Schnirelman theory, there exists a sequence of positive eigenvalues of problem \eqref{eq:1.1} with $A=-2\Delta_p$ (see, e.g., \cite{Le}), but
this sequence may not constitute the whole eigenvalue set.

Note that the (nonhomogeneous) case
$$
p\in (1, \infty), \ q=2, \ p \neq q
$$
has been considered recently by Abreu and Madeira in \cite{AM}
where the reader can also find some useful historical comments. They assume weaker conditions on $a$ and $b$. In this paper we extend their result to
the case $q>2$ but we restrict ourselves to functions $a \in L^{\infty}(\Omega), \ b\in L^{\infty}(\partial \Omega)$ since assuming weaker regularity for these
functions leads to similar results without essential changes. Note that the case
$$
p\in (1,\infty), \ q\ge 2, \ p\neq q, \ a\equiv 1, \ b\equiv 0
$$
has been solved in three previous papers, \cite{MMih}, \cite{FMS}, \cite{MM}. All these previous contributions are particular cases of the main result of this
paper (Theorem 3.1).

\section{Preliminary results}

Our hypotheses $H_{pq}$, $(H_{\Omega})$, $(H_{ab})$ will be assumed throughout this paper. If we choose $v=u_\lambda$
in \eqref{eq:1.3} (see Definition~\ref{def1}) we observe that the eigenvalues of problem \eqref{eq:1.1} cannot be negative numbers.
It is also obvious that $\lambda_0=0$ is an eigenvalue of this problem and the corresponding eigenfunctions
are the nonzero constant functions. So any other eigenvalue belongs to $(0,\infty)$.

If we assume that $\lambda>0$ is an eigenvalue of problem \eqref{eq:1.1} and choose $v\equiv 1$ in \eqref{eq:1.3} we deduce that every eigenfunction
$u_{\lambda}$ corresponding to $\lambda$ satisfies the equation
\begin{equation}\label{eq:2.1}
\int_\Omega a\mid  u_\lambda\mid ^{q-2} u_\lambda  ~dx+\int_{\partial\Omega} b \mid  u_\lambda\mid ^{q-2} u_\lambda  ~d\sigma=0.
\tag{2.1}
\end{equation}
So all eigenfunctions corresponding to positive eigenvalues necessarily belong to the set
\[
\mathcal{C}:=\Big\{ u\in W;~\int_\Omega a\mid  u\mid ^{q-2} u  ~dx+\int_{\partial\Omega} b \mid  u\mid ^{q-2} u  ~d\sigma=0\Big\}.
\]
This is a symmetric cone and using the Lebesgue Dominated Convergence Theorem (see also \cite[Theorem 4.9] {Br}) we can see
that $\mathcal{C}$ is a weakly closed subset of $W$. In addition, $\mathcal{C}$ has nonzero elements. To show this, we first note that \eqref{eq:1.2}
implis that either $|\{x\in \Omega;~a(x)> 0\}|_N>0$ or $a=0$ a.e. in $\Omega$ and $|\{x\in \partial\Omega;~b(x)> 0\}|_{N-1}>0$,
where $|\cdot |_N$ and $|\cdot |_{N-1}$ denote the Lebesgue measures of the two sets. In the former case we choose $x_1, x_2\in \Omega,~ x_1\neq x_2$,
$r>0$, such that $B_r(x_1)\cap B_r(x_2)=\emptyset,~B_r(x_k)\subset \Omega$, $|\{x\in B_r(x_k);~a(x)> 0\}|_N>0,\, k=1,2$,
and consider the test functions $u_k: \Omega\rightarrow \mathbb{R}, \, k=1,2,$
\begin{equation*}
u_k(x)=
\left\{\begin{array}{l}
e^{-\frac{1}{r^2-\mid x-x_k\mid^2}},~\mbox{if}~x\in B_r(x_k),\\[1mm]
0, \hspace*{1.8cm}\mbox{otherwise.}
\end{array}\right.
\end{equation*}
Clearly $u_k\in W$, $k=1,2$. Denote
\[
\theta_k=\int_\Omega a  u_k ^{q-1}   ~dx+\int_{\partial\Omega} b  u_k ^{q-1}  ~d\sigma.
\]
Obviously $\theta_k>0$, $k=1,2$. Define $\sigma_k=\theta_k^{\frac{-1}{q-1}},~~ k=1,2$.
It is then easily seen that the function $v=\sigma_1 u_1-\sigma_2 u_2$ belongs to $\mathcal{C}\setminus \{0\}$. Of course, $tv\in \mathcal{C}$ for all
$t\in \mathbb{R}$. A similar construction can be used
in the later case, where restrictions of similar test functions to $B_r(x_k)\cap \partial \Omega$, $x_k\in \partial \Omega$, $k=1,2$, can be considered.

\begin{rem}\label{remark11}
If for some $\lambda>0$ $u\in W\setminus\{0\}$ satisfies the equation
\[
\int_\Omega \Big(\mid \nabla u\mid ^{p}+\mid \nabla u\mid ^{q}\Big)~dx=\lambda\Big(\int_\Omega a\mid  u\mid ^{q} ~dx+\int_{\partial\Omega} b
\mid  u\mid ^{q} ~d\sigma\Big),
\]
then $u$ cannot be a constant function (see \eqref{eq:1.2}) and so
\[
\int_\Omega a\mid  u\mid ^{q} ~dx+\int_{\partial\Omega} b \mid  u\mid ^{q} ~d\sigma> 0.
\]
Therefore, denoting $\Gamma_1(u):=\{x\in \Omega;~a(x)u(x)\neq 0\},~\Gamma_2(u):=\{x\in \partial\Omega;~b(x)u(x)\neq 0\}$, we see that
either $|\Gamma_1(u)|_N>0$ or $|\Gamma_2(u)|_{N-1}>0$.

Obviously $u_{\lambda}$ corresponding to any eigenvalue $\lambda >0$ cannot be a constant function (see \eqref{eq:1.3} with $v=u_\lambda$
and \eqref{eq:1.2}).
 \end{rem}

Now, for $r>1$ define the set
\[
\mathcal{C}_r:=\Big\{ u\in W^{1,r}(\Omega);~\int_\Omega a\mid  u\mid ^{r-2} u  ~dx+\int_{\partial\Omega} b \mid  u\mid ^{r-2} u  ~d\sigma=0\Big\}.
\]
Arguing as before, we infer that for all $r>0$ $\mathcal{C}_r$ is a symmetric, weakly closed (in $W^{1,r}(\Omega)$) cone, containing infinitely
many nonzero elements.

Note also that $\mathcal{C}=\mathcal{C}_q$ if $q>p$, otherwise (i.e., if $q<p$) $\mathcal{C}$ is a proper subset of $\mathcal{C}_q.$

\bigskip

Now let us define,
\[
\mathcal{C}_{1q}:=\mathcal{C}_{q}\cap\Big\{ u\in W^{1,q}(\Omega);\int_\Omega a\mid  u\mid ^{q} ~dx+\int_{\partial\Omega} b
\mid  u\mid ^{q} ~d\sigma=1\Big\}.
\]
This set is nonempty. Indeed, let us suppose that $|\{ x\in \Omega; \, a(x)>0 \}|_N >0$ and choose $v= \sigma_1u_1 - \sigma_2u_2$ as before. We have $v\in C_q$ and
$\int_{\Omega}a|v|^q \, dx >0$ so there exists a $t_* >0$ such that
\[
t_*^q \Big( \int_\Omega a\mid  v\mid ^{q} ~dx+\int_{\partial\Omega} b
\mid  v\mid ^{q} ~d\sigma \Big)=1 \, .
\]
Therefore $t_*v\in \mathcal{C}_{1q}$. A similar conclusion is obtained if $a=0$ a.e. in $\Omega$ but $| \{ x\in \partial \Omega; \, b(x)>0\} |_{N-1}>0$.

Consider the minimization problem
\begin{equation}\label{eq:2.2}
\underset{w\in\mathcal{C}_{1q}}{\inf }~J(w) \, ,
\tag{2.2}
\end{equation}
where $J: W^{1,q}(\Omega)\rightarrow \mathbb{R}$ is defined by  $J(w):=\int_\Omega \mid  \nabla w\mid ^{q} ~dx.$ Functional $J$ is positively
homogeneous of order $q$, convex and weakly lower semicontinuous. The next result states that $J$ attains its minimal value and this value is positive.

\begin{lem}\label{lema1}
For each $q>1$ there exists $u^{*}\in \mathcal{C}_{1q}$ such that $J(u^{*})=\underset{w\in\mathcal{C}_{1q}}{\inf }~J(w)>0.$
\end{lem}
\begin{proof}
It is well-known that functional $J$ is of class $C^1$ on $W^{1,q}(\Omega)$ and obviously $J$ is bounded below. Let $(u_n)\subset \mathcal{C}_{1q}$
be a minimizing sequence for $J$, i. e.,
$$
J(u_n)\rightarrow \underset{w\in\mathcal{C}_{1q}}{\inf }~J(w):=\sigma.
$$
We can prove that $(u_n)$ is bounded in $W^{1,q}(\Omega)$. Assume the contrary, that there exists a subsequence of $(u_n)$, again
denoted $(u_n)$, such that $\parallel u_n\parallel_{L^q(\Omega)}\rightarrow\infty$ as $n\rightarrow\infty.$ Define
$$
v_n=\frac{u_n}{\parallel u_n\parallel_{L^q(\Omega)}} \ \ \ \forall~n\in \mathbb{N}\, .
$$
Clearly sequence $(v_n)$ is bounded in $ W^{1,q}(\Omega)$ so there exist a $v\in W^{1,q}(\Omega)$ and a subsequence of $(v_n)$, again denoted
$(v_n)$, such that
\[
v_n\rightharpoonup v ~\mbox{in} ~W^{1,q}(\Omega),
\]
\[
v_n\rightarrow v ~\mbox{in} ~L^{q}(\Omega),~v_n\rightarrow v ~\mbox{in} ~L^{q}(\partial\Omega).
\]
As $\parallel v_n\parallel_{L^q(\Omega)}=1~\forall ~n\in \mathbb{N}$ we have $\parallel v\parallel_{L^q(\Omega)}=1$, and
\[
\int_\Omega \mid \nabla v\mid ^{q}~dx\leq \underset{n\rightarrow \infty}{\liminf}~\int_\Omega \mid \nabla v_n\mid ^{q}~dx=\underset{n\rightarrow \infty}{\liminf}\frac{1}{\parallel u_n\parallel^q_{L^q(\Omega)}}J(u_n)=0,
\]
which shows that $v$ is a constant function. On the other hand, since $(v_n)\subset \mathcal{C}_{q}$ and $\mathcal{C}_{q}$ is weakly closed in
$W^{1,q}(\Omega)$, we infer that $v\in \mathcal{C}_{q}$, hence $v\equiv 0$. But this contradicts the fact that $\parallel v\parallel_{L^q(\Omega)}=1$.
Therefore, $(u_n)$ is indeed bounded in $W^{1,q}(\Omega)$, hence there exist $u^*\in W^{1,q}(\Omega)$ and a subsequence of $(u_n)$, which is also
denoted $(u_n)$, such that
\[
u_n\rightharpoonup u^* ~\mbox{in} ~W^{1,q}(\Omega),
\]
\[
u_n\rightarrow u^* ~\mbox{in} ~L^{q}(\Omega),~u_n\rightarrow u^* ~\mbox{in} ~L^{q}(\partial\Omega).
\]
By Lebesgue's Dominated Convergence Theorem we obtain $u^*\in \mathcal{C}_{1q}$, so the weak lower semicontinuity of $J$ leads to $\sigma=J(u^*).$
In addition $J(u^*)>0$. Indeed, assuming by contradiction that $J(u^*)=0$ would imply that $u^* \equiv Const.$, which is impossible because $u^*\in
\mathcal{C}_{1q}$.
\end{proof}

\begin{rem}\label{remarca}
For $p, \, q, \, \Omega$ satisfying our assumptions define
\begin{equation}\label{eq:2.3}
\lambda_1:=\underset{w\in\mathcal{C}\setminus\{0\}}{\inf }~\frac{\int_\Omega\mid\nabla w\mid^q~dx}{\int_{\Omega}a\mid w\mid^q~dx+
\int_{\partial\Omega}b\mid  w\mid^q~d\sigma},
\tag{2.3}
\end{equation}
and
\begin{equation}\label{eq:2.4}
\widetilde{\lambda}_1:=\underset{w\in\mathcal{C}\setminus\{0\}}{\inf }~\frac{\frac{1}{q}\int_\Omega\mid\nabla w\mid^q~dx+\frac{1}{p}\int_\Omega\mid\nabla w\mid^p~dx}{\frac{1}{q}\big(\int_{\Omega}a\mid w\mid^q~dx+\int_{\partial\Omega}b\mid w\mid^q~d\sigma\big)}.
\tag{2.4}
\end{equation}
Note that the denominators of the above fractions may equal zero for some $w$'s in $\mathcal{C}\setminus \{ 0\}$ and in such cases the corresponding
numerators are obviously $\neq 0$ thus the values of those fractions are considered $\infty$ so they do not contribute to $\lambda_1$
or $\widetilde{\lambda}_1$.

In fact $\lambda_1=\widetilde{\lambda}_1.$ Indeed, it is obvious that $\lambda_1\leq\widetilde{\lambda}_1$ and for the converse inequality we
note that $\forall v\in
\mathcal{C}\setminus\{0\}$, $t>0,$ we have $tv\in \mathcal{C}\setminus\{0\}$ and
\begin{equation*}
\begin{split}
\widetilde{\lambda}_1=&\underset{w\in\mathcal{C}\setminus\{0\}}{\inf }~\frac{\frac{1}{q}\int_\Omega\mid\nabla w\mid^q~dx+
\frac{1}{p}\int_\Omega\mid\nabla w\mid^p~dx}{\frac{1}{q}\big(\int_{\Omega}a\mid w\mid^q~dx+\int_{\partial\Omega}b\mid w\mid^q~d\sigma\big)}\leq\\
&\frac{\int_\Omega\mid\nabla v\mid^q~dx}{\int_{\Omega}a\mid v\mid^q~dx+\int_{\partial\Omega}b\mid v\mid^q~d\sigma}+
t^{p-q}\frac{q\int_\Omega\mid\nabla v\mid^p~dx}{p\big(\int_{\Omega}a\mid w\mid^q~dx+\int_{\partial\Omega}b\mid v\mid^q~d\sigma\big)}.
\end{split}
\end{equation*}
Now letting $t\rightarrow\infty$ if $q>p$, and $t\rightarrow 0$ if $q<p$, then passing to infimum for $v\in \mathcal{C}\setminus\{0\}$
we get the desired inequality. Therefore $\lambda_1$ can be expressed in two different ways (see \eqref{eq:2.3} and \eqref{eq:2.4}).
\end{rem}

\begin{rem}\label{remark2}
As a consequence of Lemma~\ref{lema1} we have $\lambda_1>0$. Indeed,
\[
\lambda_1:=\underset{w\in\mathcal{\widetilde{C}}}{\inf }~\int_{\Omega}\mid\nabla w\mid^q~dx \, ,
\]
where $\mathcal{\widetilde{C}}=\{v\in \mathcal{C}; \int_{\Omega}a\mid\nabla v\mid^q~dx+\int_{\partial\Omega}b\mid\nabla v\mid^q~d\sigma=1\}$. So $\lambda_1=J(u^*)$ for $p\leq q$
  and $\lambda_1\geq J(u^*)$ if $p > q.$ Thus in both cases $\lambda_1>0.$
\end{rem}
Now we recall a result which is known as the Lagrange multiplier rule (see, e.g., \cite[Thm. 2.2.18, p. 78]{papa}).
\begin{lem}\label{lema2}
Let $X, Y$ be real Banach spaces and let $f:D\rightarrow \mathbb{R}$ be Fr\'{e}chet differentiable , $g\in C^1(D,Y)$, where $D\subseteq X$ is
a nonempty open set. If $v_0$ is a local minimizer of the constraint problem
\[
{\min }~f(v), \ \ \ g(v)=0,
\] and $\mathcal{R}(g'(v_0))$ (the range of $g'(v_0)$) is closed, then there exist $\lambda^*\in \mathbb{R}$ and $y^{*}\in Y^{*}$ not both equal
to zero such that
\begin{equation}\label{eq:2.5}
\lambda^*f'(v_0)+y^{*}\circ g'(v_0)=0,
\tag{2.5}
\end{equation}
where $Y^{*}$ stands for the dual of $Y.$
\end{lem}
\begin{rem}\label{remark13}
Define
\begin{equation}\label{eq:2.6}
\parallel u\parallel_{ab}:=\parallel \nabla u\parallel_{L^p(\Omega)}+\Big(\parallel a^{1/q} u\parallel^q_{L^q(\Omega)}+
\parallel b^{1/q} u\parallel^q_{L^q(\partial\Omega)}\Big)^{1/q} \ \ \forall u\in W^{1,p}(\Omega)
 \tag{2.6}
\end{equation}
If $p>q$ and $a, \, b$ satisfy $(H_{ab})$ then \eqref{eq:2.6} is a norm in $W^{1,q}(\Omega)$ equivalent with the usual norm of this space. 
This fact follows from \cite[Proposition 3.9.55]{DMP}. Indeed, the seminorm
\[
w(u):= \Big(\parallel a^{1/q} u\parallel^q_{L^q(\Omega)}+
\parallel b^{1/q} u\parallel^q_{L^q(\partial\Omega)}\Big)^{1/q} \ \ \forall u\in W^{1,p}(\Omega),
\]
satisfies the two requirements of that proposition \\
(i)~~$\exists d>0$ such that $w(u)\leq d \parallel u\parallel_{W^{1,p}(\Omega)} \ \ \forall u\in W^{1,p}(\Omega),$ and\\
(ii)~if $u=\mbox{constant}$, then $w(u)=0$ implies $u\equiv 0.$
\end{rem}

\section{ The main result}

Let us state the main result of this paper:

\begin{thm}\label{teorema1}
Assume that $(H_{pq})$, $(H_{\Omega})$ and $(H_{ab})$ above are fulfilled. Then the set of eigenvalues of problem \eqref{eq:1.1}
is $\{0\} \cup (\lambda_1, \infty)$, where $\lambda_1$ is the positive constant defined by \eqref{eq:2.3}.
\end{thm}
\begin{proof}
We have alredy said that $\lambda_0=0$ is an eigenvalue of problem \eqref{eq:1.1} and any other eigenvalue of this problem belongs to $(0,\infty)$.
Let us first prove that there is no eigenvalue of problem \eqref{eq:1.1} in $(0, \lambda_1]$. Assume by contradiction that there exists an
eigenvalue $\lambda\in (0, \lambda_1]$ and let $u_\lambda\in \mathcal{C}\setminus\{0\}$ be a corresponding eigenfunction. Choosing $v=u_\lambda$
in \eqref{eq:1.3} yields
\begin{equation}\label{eq:3.1}
\int_\Omega \big(\mid \nabla u_\lambda\mid ^{p}+\mid \nabla u_\lambda\mid ^{q}\big)~dx=\lambda\Big(
\int_\Omega a\mid  u_\lambda\mid ^{q}~dx+\int_{\partial\Omega} b \mid  u_\lambda\mid ^{q} ~d\sigma
\Big).
\tag{3.1}
\end{equation}
Note that $\int_\Omega a\mid  u_\lambda\mid ^{q}~dx+\int_{\partial\Omega} b \mid  u_\lambda\mid ^{q} ~d\sigma \neq 0$, otherwise $u_{\lambda}\equiv Const.$
(cf. \eqref{eq:3.1}) which is impossible (see Remark \ref{remark11}).
On the other hand, as $u_\lambda\in \mathcal{C}\setminus\{0\}$, we derive from \eqref{eq:2.3} and \eqref{eq:3.1}
\begin{equation*}
\begin{split}
\lambda\leq \lambda_1\leq &\frac{\int_\Omega\mid\nabla u_\lambda\mid^q~dx}{\int_{\Omega}a\mid
 u_\lambda \mid^q~dx+\int_{\partial\Omega}b\mid u_\lambda\mid^q~d\sigma}\\
&=\frac{\lambda\Big(\int_{\Omega}a\mid u_\lambda \mid^q~dx+\int_{\partial\Omega}b\mid u_\lambda\mid^q~d\sigma\Big)-
\int_\Omega\mid \nabla u_\lambda\mid^p~dx}{\int_{\Omega}a\mid u_\lambda \mid^q~dx+\int_{\partial\Omega}b\mid u_\lambda\mid^q~d\sigma}\\
&= \lambda-\frac{\int_\Omega\mid\nabla u_\lambda\mid^p~dx}{\int_{\Omega}a\mid u_\lambda \mid^q~dx+
\int_{\partial\Omega}b\mid u_\lambda\mid^q~d\sigma}< \lambda,
\end{split}
\end{equation*}
which is clearly impossible.

In what follows we shall prove that every $\lambda>\lambda_1$ is an eigenvalue of problem \eqref{eq:1.1}. To this purpose we fix such a $\lambda$ and
define the functional $\mathcal{J}_\lambda: W\rightarrow \mathbb{R},$
\[
\mathcal{J}_\lambda(u)=\frac{1}{p}\int_\Omega \mid \nabla u\mid ^{p}~dx+\frac{1}{q}\int_\Omega\mid \nabla u\mid ^{q}~dx-
\frac{\lambda}{q}\Big(\int_\Omega a\mid  u\mid ^{q}~dx+\int_{\partial\Omega} b \mid  u_\lambda\mid ^{q} ~d\sigma\Big).
\]
It is easily seen that functional $\mathcal{J}_\lambda\in C^1(W\setminus\{0\};\mathbb{R})$ (even more, $\mathcal{J}_\lambda\in C^1(W;\mathbb{R})$
if $2<q<p$)
and
\begin{equation*}
\begin{split}
\langle\mathcal{J}'_\lambda(u),v\rangle&=\int_\Omega \mid \nabla u\mid ^{p-2}\nabla u\cdot\nabla v~dx+
\int_\Omega\mid \nabla u\mid ^{q-2}\nabla u\cdot\nabla v~dx\\
&-\lambda\Big(\int_\Omega a\mid  u\mid ^{q-2}uv~dx+\int_{\partial\Omega} b \mid  u_\lambda\mid ^{q-2}uv ~d\sigma\Big)\ \
\forall v\in W,~ u\in W\setminus\{0\}.
\end{split}
\end{equation*}
So, according to Definition~\ref{def1}, $\lambda>\lambda_1$ is an eigenvalue of problem \eqref{eq:1.1} if and only if there exists
a critical point $u_\lambda\in W\setminus\{0\}$ of $\mathcal{J}_\lambda$, i. e.  $\mathcal{J}'_\lambda(u_\lambda)=0$.

We shall discuss two cases which are complementary to each other.

\bigskip

\textbf{Case 1: $2<q<p$}. We shall prove that in this case functional $\mathcal{J}_\lambda$ is coercive on
$\mathcal{C}\subset W=W^{1,p}(\Omega)$, i. e.,
\[
\underset{\parallel u\parallel_{W^{1,p}(\Omega)}\rightarrow\infty, u\in\mathcal{C}}{\lim}\mathcal{J}_\lambda(u)=\infty.
\]
To this purpose we define $T_1, T_2, T_3: \mathcal{C}\rightarrow\mathbb{R}$ as follows
$$
T_1(u)=\int_\Omega \mid \nabla u\mid ^{p}~dx,~~ T_2(u)=\int_\Omega\mid \nabla u\mid ^{q}~dx, ~~T_3(u)=
\int_\Omega a\mid  u\mid ^{q}~dx+\int_{\partial\Omega} b \mid  u_\lambda\mid ^{q} ~d\sigma,
$$
so $\mathcal{J}_\lambda(u)=\frac{1}{p}T_1(u)+\frac{1}{q}T_2(u)-\frac{\lambda}{q}T_3(u).$

We know from Remark~\ref{remark13} that the usual norm of $W^{1,p}(\Omega),$ denoted $\parallel \cdot \parallel_{W^{1,p}(\Omega)}$,
is equivalent with the norm $\parallel \cdot \parallel_{ab}$ defined in \eqref{eq:2.6}. Thus $\parallel u\parallel_{W^{1,p}(\Omega)}\rightarrow\infty$
if and only if $\parallel u\parallel_{ab}=T_1(u)^{1/p}+T_3(u)^{1/q}\rightarrow\infty$. From \eqref{eq:2.3} we then have
\begin{equation}\label{eq:3.2}
 \lambda_1 T_3(u)\leq T_2(u)~~\forall~ u\in \mathcal{C},
\tag{3.2}
\end{equation}
 hence
 \[
 \frac{1}{p}T_1(u)+\frac{1}{q}T_2(u)\geq \frac{1}{p}\big(T_1(u)+T_2(u)\big)\geq  \frac{1}{p}(1+\lambda_1)\big(T_1(u)+T_3(u)\big),
 \]
which implies
\begin{equation}\label{eq:3.3}
\underset{\parallel u\parallel_{W^{1,p}(\Omega)}\rightarrow\infty, u\in\mathcal{C}}{\lim}\biggl(\frac{1}{p}T_1(u)+\frac{1}{q}T_2(u)\biggr)=\infty.
\tag{3.3}
\end{equation}
By H\"{o}lder's inequality we have,
\[
T_2(u)\leq\mid\Omega\mid_N^{(p-q)/p} T_1(u)^{q/p}~~\forall~u\in W^{1,p}(\Omega),
\]
so it follows from \eqref{eq:3.3}
\begin{equation}\label{eq:3.4}
\underset{\parallel u\parallel_{W^{1,p}(\Omega)}\rightarrow\infty, u\in\mathcal{C}}{\lim}T_1(u)=\infty.
\tag{3.4}
\end{equation}
So, we obtin from \eqref{eq:3.2} and H\"{o}lder's inequality
\[
\mathcal{J}_\lambda(u)\geq \frac{1}{p}T_1(u)+\frac{1}{q}T_2(u)-\frac{\lambda}{\lambda_1 q}T_2(u)\geq \frac{1}{p}T_1(u)-
\frac{\lambda}{\lambda_1 q}\mid\Omega\mid^{(p-q)/p} T_1(u)^{q/p}.
\]
Since $q<p$ the right-hand side of the above inequality tends to $\infty$ as $\parallel u\parallel_{W^{1,p}(\Omega)}\rightarrow\infty$ (cf. \eqref{eq:3.4})
so $\mathcal{J}_\lambda$ is indeed coercive on $\mathcal{C}$.

We note that $\mathcal{C}$ is a weakly closed subset of the reflexive Banach space $W=W^{1,p}(\Omega),$ and functional $\mathcal{J}_\lambda$
is weakly lower semicontinuous on $\mathcal{C}$ with respect to the norm of $W^{1,p}(\Omega).$
So $\mathcal{J}_\lambda$ has a global minimizer $u_*\in \mathcal{C}$, i.e., $\mathcal{J}_\lambda(u_*)=\min_{\mathcal{C}}\mathcal{J}_\lambda$
(see, e.g., \cite [Theorem 1.2] {St}). From Remark \ref{remarca} we know that $\lambda_1=\widetilde{\lambda}_1$, hence \eqref{eq:2.4})
as $\lambda>\lambda_1=\widetilde{\lambda}_1$. Then (by \eqref{eq:2.4}) there exists $u_{0\lambda}\in \mathcal{C}$ such that
$\mathcal{J}_\lambda(u_{0\lambda})<0.$ It follows that
\[
\mathcal{J}_\lambda(u_*)\leq \mathcal{J}_\lambda(u_{0\lambda})<0,
\]
which shows that $u_*\neq 0.$ In fact $u_*$ is a solution of the minimization problem
\[
\min_{v\in W} \mathcal{J}_\lambda(v),
\]
under the restriction
\[
g(v):=\int_\Omega a\mid  v\mid ^{q-2} v  ~dx+\int_{\partial\Omega} b \mid  v\mid ^{q-2} v  ~d\sigma=0.
\]
We can apply Lemma~\ref{lema2} with $X=W,~ D=W,~Y=\mathbb{R}, f=\mathcal{J}_\lambda,$  $g:W\rightarrow\mathbb{R}$ being the function just defined above,
and $v_0=u_*,$ on the condition that $\mathcal{R}(g'(u_*))$ is a closed set. In fact we can show that $g'(u_*)$ is surjective, i.e.,
$\forall~~ c\in \mathbb{R}$ there exists a $w\in W$ such that
\[
\langle g'(u_*),w\rangle=c.
\]
We seek $w$ of the form $w= u_*+\beta,~\beta \in \mathbb{\mathbb{R}} $. Thus we obtain from the above equation (using $u_*\in \mathcal{C}$)
\[
\beta(q-1)\Big(\int_\Omega a\mid  u_*\mid ^{q-2} ~dx+\int_{\partial\Omega} b \mid  u_*\mid ^{q-2}   ~d\sigma\Big)=c,
\]
which has a unique solution $\beta$ since
\[
\int_{\Omega} a\mid u_*{\mid}^{q-2}~dx + \int_{\partial \Omega}b \mid u_*{\mid}^{q-2} ~ d\sigma \neq 0 \, ,
\]
otherwise $\mathcal{J}_{\lambda}(u_*)=p^{-1}\int_{\Omega}\mid \nabla u_*{\mid}^p +q^{-1}\int_{\Omega} \mid \nabla u_*{\mid}^q dx$ which
contradicts $\mathcal{J}_{\lambda}(u_*) <0$. Thus $g'(u_*)$ is surjective, as claimed. By Lemma~\ref{lema2} there exist $\lambda^*, \mu\in \mathbb{R}$,
not both equal to zero, such that
\[
\lambda^*\langle \mathcal{{J}}'_\lambda(u_*), v\rangle+\mu\langle g'(u_*), v\rangle=0,~~\forall~v\in W=W^{1,p}(\Omega),
\]
or, equivalently,
\begin{equation*}
\begin{split}
\lambda^*\biggl(\int_\Omega \Big(\mid \nabla u_*\mid ^{p-2}+&\mid \nabla u_*\mid ^{q-2}\Big)\nabla u_* \cdot \nabla v~dx, \\
&-\lambda\Big(\int_\Omega a\mid  u_*\mid ^{q-2} u_*  v~dx+\int_{\partial\Omega} b \mid  u_*\mid ^{q-2} u_*  v~d\sigma\Big)\biggr)\\
&+\mu (q-1)\Big(\int_\Omega a\mid  u_*\mid ^{q-2}  v~dx+\int_{\partial\Omega} b \mid  u^*\mid ^{q-2}   v~d\sigma\Big)=0 \ \ \forall v\in W.
\end{split}
\end{equation*}
Choosing $v\equiv 1$ in the above equality and taking into account the fact that $u_*\in \mathcal{C}$ we get
\[
\mu \Big(\int_\Omega a\mid  u_*\mid ^{q-2}  ~dx+\int_{\partial\Omega} b \mid  u_*\mid ^{q-2}   ~d\sigma\Big)=0,
\]
which implies $\mu=0$. Threfore $\lambda^*\neq 0$ and so
\begin{equation*}
\begin{split}
\int_\Omega \Big(\mid \nabla u_*\mid ^{p-2}+&\mid \nabla u_*\mid ^{q-2}\Big)\nabla u_* \cdot \nabla v~dx, \\
&-\lambda\Big(\int_\Omega a\mid  u_*\mid ^{q-2} u_*  v~dx+\int_{\partial\Omega} b \mid  u_*\mid ^{q-2} u_*  v~d\sigma\Big)=0~\forall~v\in W,
\end{split}
\end{equation*}
i. e., $\lambda$ is an eigenvalue of problem \eqref{eq:1.1}

\bigskip

\textbf{Case 2: $q> 2$, $1<p<q$}. In this case $W=W^{1,q}(\Omega).$ Let $\lambda>\lambda_1$ be a fixed number. In this case we cannot expect coercivity on $W$ for functional $\mathcal{J}_\lambda$  which obviously belongs to $C^1(W\setminus\{0\};\mathbb{R}).$ We shall prove that $\mathcal{J}_\lambda$ has a critical point in $\mathcal{C}\setminus\{0\}$. To this purpose we consider 
a Nehari type manifold (see \cite{SW}):
\begin{equation*}
\begin{split}
\mathcal{N}_\lambda&=\{v\in \mathcal{C}\setminus\{0\}; \langle \mathcal{J}'_\lambda(v),v\rangle=0\}\\
&=\Big\{v\in \mathcal{C}\setminus\{0\}; \int_\Omega \big(\mid \nabla v\mid ^{p}+\mid \nabla v\mid ^{q}\big)~dx
=\lambda\Big(\int_\Omega a\mid  v\mid ^{q} ~dx+\int_{\partial\Omega} b \mid  v\mid ^{q} ~d\sigma\Big)\Big\}.
\end{split}
\end{equation*}
It is natural to consider the restriction of $\mathcal{J}_\lambda$ to $\mathcal{N}_\lambda$ as any possible eigenfunction corresponding to $\lambda$ belongs to $\mathcal{N}_\lambda$. Note that 
on $\mathcal{N}_\lambda$ functional $\mathcal{J}_\lambda$ has the form
\begin{equation*}
\begin{split}
\mathcal{J}_\lambda(u)&=\frac{1}{p}\int_\Omega \mid \nabla u\mid ^{p}~dx+\frac{1}{q}\int_\Omega\mid \nabla u\mid ^{q}~dx-\frac{\lambda}{q}\Big(\int_\Omega a\mid  u\mid ^{q}~dx+\int_{\partial\Omega} b \mid  u\mid ^{q} ~d\sigma\Big)\\
&=\frac{1}{p}\int_\Omega \mid \nabla u\mid ^{p}~dx-\frac{1}{q}\int_\Omega\mid \nabla u\mid ^{q}~dx=\frac{q-p}{qp}\int_\Omega \mid \nabla u\mid ^{p}~dx>0.
\end{split}
\end{equation*}
We shall prove that there exists a point $u_*\in \mathcal{N}_\lambda$ where $\mathcal{J}_\lambda$ attains its minimal value, $m_\lambda:= \underset{w\in\mathcal{{N}}_\lambda}{\inf }{\mathcal{{J}}_\lambda (w) }$ and $\mathcal{J}'_\lambda(u^*)=0.$ The proof relies on essentially known and new arguments, and is divided into several steps as follows. 

\bigskip

\textbf{Step 1.} $\mathcal{{N}}_\lambda \neq \emptyset.$

Indeed, since $\lambda>\lambda_1,$ we deduce from \eqref{eq:2.3} that there exists a $v_0\in \mathcal{C}\setminus\{0\} $ such that
\[
\int_\Omega\mid\nabla v_0\mid^q~dx < \lambda \Big(\int_{\Omega}a\mid v_0\mid^q~dx+\int_{\partial\Omega}b\mid v_0\mid^q~d\sigma\Big).
\]
We claim that for a convenient $t>0$, $tv_0\in \mathcal{{N}}_\lambda$. Since $\mathcal{{C}}$ is a cone, $tv_0\in \mathcal{{C}}$ for all $t\in \mathbb{R}$. So the condition 
$tv_0\in \mathcal{N}_\lambda$, $t>0$, reads
\[
t^p\int_\Omega \mid \nabla v_0\mid ^{p}~dx+t^q\int_\Omega\mid \nabla v_0\mid ^{q}~dx
=\lambda t^q\Big(\int_\Omega a\mid  v_0\mid ^{q} ~dx+\int_{\partial\Omega} b \mid  v_0\mid ^{q} ~d\sigma\Big). 
\]
This equation can be solved for $t$, 
\begin{equation}\label{eq:3.5}
t=\Biggl(\frac{\int_\Omega \mid \nabla v_0\mid ^{p}~dx }{\lambda \big(\int_{\Omega}a\mid v_0\mid^q~dx+\int_{\partial\Omega}b\mid v_0\mid^q~d\sigma\big)-\int_\Omega\mid\nabla v_0\mid^q~dx}\Biggr)^{1/(q-p)}, 
\tag{3.5}
\end{equation}
and hence for this $t$ we have $tv_0\in \mathcal{N}_\lambda$.

\bigskip
 
\textbf{Step 2.} Every minimizing sequence $(u_n)\subset\mathcal{{N}}_\lambda$ for $\mathcal{{J}}_\lambda$ is bounded in $W=W^{1,q}(\Omega). $

Let $(u_n)\subset\mathcal{{N}}_\lambda$ be such a minimizing sequence for $\mathcal{{J}}_\lambda$. Since $u_n\in \mathcal{{N}}_\lambda$ for all $n$, we have 
\begin{equation}\label{eq:3.6}
\mathcal{J}_\lambda(u_n)=\frac{q-p}{qp}\int_\Omega \mid \nabla u_n\mid ^{p}~dx\rightarrow  m_\lambda,~\mbox{as}~n\rightarrow\infty,
\tag{3.6}
\end{equation}
and 
\begin{equation}\label{eq:3.7}
\begin{split}
0&<\lambda\Big( \int_\Omega a\mid  u_n\mid ^{q} ~dx+\int_{\partial\Omega} b \mid  u_n\mid ^{q} ~d\sigma\Big)-\int_\Omega\mid \nabla u_n\mid ^{q}~dx\\
&=\int_\Omega \mid \nabla u_n\mid ^{p}~dx\rightarrow \frac{qp}{q-p} m_\lambda,~\mbox{as}~n\rightarrow\infty.
\end{split}
\tag{3.7}
\end{equation}
Assume by contradiction that  $(u_n)$ is unbounded in $W^{1,q}(\Omega).$ Then, on a subsequence, again denoted $(u_n)$, we have $\Vert u_n\Vert_{ab} \rightarrow \infty$ (for details on $\Vert \cdot \Vert_{ab}$ see Remark \ref{remark13}).  It follows from \eqref{eq:3.7} that (on a subsequence)
\[c_{n}:=\big(\parallel a^{1/q}u_n\parallel_{L^q(\Omega)}^q+\parallel b^{1/q} u_n\parallel_{L^q(\partial\Omega)}^q\big)^{1/q} \rightarrow \infty . 
\]
Denote $v_n=u_n/c_{n},~n\in\mathbb{N}.$  From \eqref{eq:3.7} we have $\int_\Omega\mid \nabla v_n\mid ^{q}~dx\leq \lambda $ for all $n$, so $(v_n)$ is bounded with respect to the norm  $\parallel\cdot\parallel_{ab},$ which is equivalent to the usual norm of $W^{1,q}(\Omega)$.  So there exists a $v_0\in W^{1,q}(\Omega)$ such that  $v_n\rightharpoonup v_0$ in $W^{1,q}(\Omega)$ (hence also in $W^{1,p}(\Omega)$). Obviously, $v_n\rightarrow v_0$ in $L^{q}(\Omega)$ and also in $L^q(\partial\Omega).$ As $\mathcal{C}$ is weakly closed in $W$ and $(v_n)\subset \mathcal{C}$ we also have $v_0\in \mathcal{C}.$ Now, from \eqref{eq:3.7} we deduce $\int_\Omega\mid \nabla v_n\mid ^{p}~dx\rightarrow 0,$ and so 
\[
\int_\Omega\mid \nabla v_0\mid ^{p}~dx\leq \underset{n\rightarrow\infty}{\liminf}\int_\Omega\mid \nabla v_n\mid ^{p}~dx=0.
\]
Therefore $v_0$ is a constant function. In fact $v_0\equiv 0$ since $v_0\in \mathcal{C}$. It follows that $v_n\rightarrow 0$ in $L^{q}(\Omega)$ and in $L^{q}(\partial\Omega),$ which contradicts the fact that 
$$
\parallel a^{1/q}v_n\parallel_{L^q(\Omega)}^q+\parallel b^{1/q}v_n\parallel_{L^q(\partial\Omega)}^q=1~\forall ~n\in \mathbb{N}.
$$

\bigskip

\textbf{Step 3.} $m_\lambda:= \underset{w\in\mathcal{{N}}_\lambda}{\inf }{\mathcal{{J}}_\lambda (w) }>0.$

Assume that, on the contrary, $m_\lambda=0.$ Let $(u_n)\subset\mathcal{{N}}_\lambda$ be a minimizing sequence for $\mathcal{{J}}_\lambda.$ We have (see \eqref{eq:3.7}) 
\begin{equation}\label{eq:3.8}
0<\lambda\Big( \int_\Omega a\mid  u_n\mid ^{q} ~dx+\int_{\partial\Omega} b \mid  u_n\mid ^{q} ~d\sigma\Big)-\int_\Omega\mid \nabla u_n\mid ^{q}~dx=\int_\Omega \mid \nabla u_n\mid ^{p}~dx\rightarrow 0~\mbox{as}~n\rightarrow\infty.
\tag{3.8}
\end{equation}
We know from Step 2 that $(u_n)$ is bounded in $W^{1,q}(\Omega),$ so there exists $u_0\in W^{1,q}(\Omega)$ such that, on a subsequence denoted again $(u_n)$, $u_n\rightharpoonup u_0$ in 
 $W^{1,q}(\Omega)$ (hence also in $W^{1,p}(\Omega)$), and $u_n\rightarrow u_0$ in $L^q(\Omega)$, $u_n\rightarrow u_0$ in $L^q(\partial\Omega).$ Clearly $u_0\in \mathcal{C}$ and from \eqref{eq:3.8} we deduce that $u_0$ is a constant function, so $u_0\equiv 0.$ Summarizing, we have proved that $u_n\rightharpoonup 0$ in $W^{1,q}(\Omega)$.  
 As in the previous step, we define $v_n=u_n/c_{n},~n\in \mathbb{N}$. Note that $c_n>0$ for all $n$ (otherwise, by \eqref{eq:3.8} all the $u_n$'s will be constant functions, which is impossible since they belong to $\mathcal{C}\setminus \{ 0\}$).  By \eqref{eq:3.8} we see that 
 \[
 \int_{\Omega} \mid \nabla v_n {\mid}^qdx < \lambda \ \ \ \forall n\in \mathbb{N} \, , 
 \]
so $(v_n)$ is bounded in $W^{1,q}(\Omega)$. As $(v_n)$ is a sequence in $\mathcal{C}$ which is weakly closed in $W^{1,q}(\Omega)$, it follows that there exists a $v_0\in \mathcal{C}$ such that, on a subsequence, $v_n \rightharpoonup v_0$ in $W^{1,q}(\Omega)$ and $v_n \rightarrow v_0$ in $L^q(\Omega)$ as well as in $L^q(\partial \Omega)$. Now, we divide \eqref{eq:3.8} by $c_n^q$ to obtain 
\[
\int_{\Omega}\mid \nabla v_n \, {\mid}^p dx = c_n^{q-p}\big[ \lambda - \int_{\Omega} \mid \nabla v_n \, {\mid}^q dx \big] \rightarrow 0 \, . 
\]
Next, since $v_n \rightharpoonup v_0$ in $W^{1,q}(\Omega)$ (hence also in $W^{1,p}(\Omega)$), we infer that 
\[
\int_{\Omega}\mid \nabla v_0 \, {\mid}^p dx \leq \underset{n\rightarrow\infty}{\liminf}\int_\Omega\mid \nabla v_n\mid ^{p}~dx=0.
\]
Therefore $v_0$ is a constant function and in fact $v_0 \equiv 0$ since $v_0\in \mathcal{C}$. Thus, $v_n \rightarrow 0$ in both $L^q(\Omega)$ and $L^q(\partial \Omega)$. But this contradicts the fact that 
$$
\parallel a^{1/q}v_n\parallel_{L^q(\Omega)}^q+\parallel b^{1/q}v_n\parallel_{L^q(\partial\Omega)}^q=1~\forall ~n\in \mathbb{N}.
$$
This contradiction shows that $m_\lambda >0$. 

\bigskip

\textbf{Step 4.} There exists $u_*\in \mathcal{N}_\lambda$ such that $\mathcal{J}_\lambda(u_*)=m_\lambda.$

Let $(u_n)\subset\mathcal{{N}}_\lambda$  be a minimizing sequence: $\mathcal{{J}}_\lambda(u_n)\rightarrow m_\lambda.$ By Step 3 $(u_n)$ is bounded in $W^{1,q}(\Omega)$. Thus, on a subsequenc, 
$(u_n)$ converges weakly in $W^{1,q}(\Omega)$ to some $u_* \in W^{1,q}(\Omega)$ and strongly in both $L^{q}(\Omega)$ and $L^{q}(\partial\Omega)$ (to the same $u_*$). Thus, 
\begin{equation}\label{eq:3.9}
\mathcal{{J}}_\lambda(u_*)\leq \underset{n\rightarrow\infty}{\liminf} \mathcal{{J}}_\lambda(u_n)=m_\lambda.
\tag{3.9}
\end{equation}
As $(u_n)\subset\mathcal{{N}}_\lambda$ we have 
\begin{equation}\label{eq:3.10}
\int_\Omega \big(\mid \nabla u_n\mid ^{p}+\mid \nabla u_n\mid ^{q}\big)~dx
=\lambda\big(\int_\Omega a\mid  u_n\mid ^{q} ~dx+\int_{\partial\Omega} b \mid  u_n\mid ^{q} ~d\sigma\big),
\tag{3.10}
\end{equation}
\begin{equation}\label{eq:3.11}
\int_\Omega a\mid  u_n\mid ^{q-2} u_n  ~dx+\int_{\partial\Omega} b \mid  u_n\mid ^{q-2} u_n  ~d\sigma=0~\forall~\in \mathbb{N}.
\tag{3.11}
\end{equation}
It is easily seen that $u_*$ is not the null function. Indeed, assuming that $u_*\equiv 0$, we infer by \eqref{eq:3.10} that $(u_n)$ converges strongly to $0$ in $W^{1,q}(\Omega)$, hence also 
in $W^{1,p}(\Omega)$. Then \eqref{eq:3.6} will give $m_\lambda=0$ thus contradicting the statement of Step 3. Obviously $u_*\in \mathcal{{C}}\setminus \{0\}.$ Letting $n\rightarrow \infty$ in \eqref{eq:3.10} yields
\begin{equation}\label{eq:3.12}
\int_\Omega \big(\mid \nabla u_*\mid ^{p}+\mid \nabla u_*\mid ^{q}\big)~dx
\leq\lambda\big(\int_\Omega a\mid  u_*\mid ^{q} ~dx+\int_{\partial\Omega} b \mid  u_*\mid ^{q} ~d\sigma\big).
\tag{3.12}
\end{equation}
If \eqref{eq:3.12} holds with equality then we are done (cf. \eqref{eq:3.9}). We shall prove that assuming strict inequality in \eqref{eq:3.12} leads to a contradiction. Thus, let us assume that 
\begin{equation}\label{eq:3.13}
\int_\Omega \big(\mid \nabla u_*\mid ^{p}+\mid \nabla u_*\mid ^{q}\big)~dx
<\lambda\Big(\int_\Omega a\mid  u_*\mid ^{q} ~dx+\int_{\partial\Omega} b \mid  u_*\mid ^{q} ~d\sigma\Big).
\tag{3.13}
\end{equation}
If we choose $t$ as in \eqref{eq:3.5} with $u_*$ instead of $v_0$, we have $t u_*\in \mathcal{N}_\lambda$ with $t\in (0,1)$. Next, using the form of $\mathcal{J}_\lambda$ on the Nehari manifold 
$\mathcal{{N}}_\lambda$, we get
\begin{equation}\label{eq:3.14}
\mathcal{J}_\lambda(t u_*)=\frac{q-p}{qp}\int_\Omega \mid \nabla t u_*\mid ^{p}~dx=t^p\frac{q-p}{qp}\int_\Omega \mid \nabla u_*\mid ^{p}~dx.
\tag{3.14}
\end{equation}
In addition,
\[
\mathcal{J}_\lambda( u_n)=\frac{q-p}{qp}\int_\Omega \mid \nabla u_n\mid ^{p}~dx\Rightarrow m_\lambda=\underset{n\rightarrow\infty}{\lim}{\mathcal{{J}}_\lambda (u_n)}\geq\frac{q-p}{qp}\int_\Omega \mid \nabla u_*\mid ^{p}~dx.
\]
Therefore,
\[
0<m_\lambda\leq \mathcal{J}_\lambda( t u_*) =t^p\frac{q-p}{qp}\int_\Omega \mid \nabla u_*\mid ^{p}~dx\leq t^p \underset{n\rightarrow\infty}{\lim}{\mathcal{{J}}_\lambda (u_n)}=t^p m_\lambda<m_\lambda,
\]
which is impossible.

\bigskip

\textbf{Step 5.} If $u_*\in \mathcal{N}_\lambda$ is the minimizer determined in Step 4, then $\mathcal{J}'_\lambda( u_*)=0.$

In fact $u_*$  is a solution of the minimization problem
\[
\min_{v\in W} \mathcal{J}_\lambda(v),
\]
with the restrictions 
\begin{equation}\label{eq:3.15}
g_1(v):=\int_\Omega \big(\mid \nabla v\mid ^{p}+\mid \nabla v\mid ^{q}\big)~dx
-\lambda\Big(\int_\Omega a\mid  v\mid ^{q} ~dx+\int_{\partial\Omega} b \mid  v\mid ^{q} ~d\sigma\Big)=0,
\tag{3.15}
\end{equation}
\begin{equation}\label{eq:3.16}
g_2(v):=\int_\Omega a\mid  v\mid ^{q-2} v  ~dx+\int_{\partial\Omega} b \mid  v\mid ^{q-2} v  ~d\sigma=0.
\tag{3.16}
\end{equation}
We shall use again Lemma \ref{lema2}, this time with $X=W, \ Y=\mathbb{R}^2, D=W\setminus \{0\}, \ f=\mathcal{J}_\lambda,$  $g=(g_1,g_2)$ where 
$g_1, \, g_2$ are defined above,  
$x_0=u_*$. All the assumptions of Lemma \ref{lema2} can be checked easily, except the fact that $g'(u_*)$ has closed range. In fact we shall prove more, that $g'(u_*)$ is surjective, i.e., $\forall~~ (c_1, c_2)\in \mathbb{R}^2$ 
there exists a $w\in W$ such that
\[
\langle g_1'(u_*),w\rangle=c_1, ~~\langle g_2'(u_*),w\rangle=c_2.
\]
 If we choose in the above equations $w$ of the form $w=\alpha u_*+\beta,~ \alpha,\beta \in \mathbb{\mathbb{R}} $ and take into account the 
fact that $u_*\in \mathcal{N}_\lambda$, we obtain the following algebraic system
\[
\alpha(p-q)\int_\Omega \mid \nabla u_*\mid ^{p}dx=c_1,~~\beta(q-1)\Big(\int_\Omega a\mid  u_*\mid ^{q-2} ~dx+
\int_{\partial\Omega} b \mid  u_*\mid ^{q-2}   ~d\sigma\Big)=c_2,
\]
which has a unique solution $(\alpha , \beta)$ (from Remark~\ref{remark11} and $u_*\in \mathcal{N}_\lambda$ we see that the coefficients of $\alpha$ and 
$\beta$ are $\neq 0$). Thus $g'(u_*) $ is indeed surjective and so Lemma \ref{lema2} is applicable to the above constraint minimization problem. 
Therefore there exist $\lambda^* \in \mathbb{R}$ and $(\mu_1, \mu_2)\in \mathbb{R}^2$, not both equal to zero, such that 
\[
\lambda^*\langle \mathcal{{J}}'_\lambda(u_*), v\rangle+\mu_1\langle g'_1(u_*), v\rangle+\mu_2\langle g'_2(u_*), v\rangle=0,~~\forall~v\in  W^{1,q}(\Omega),
\]
or, equivalently, 
\begin{equation*}
\begin{split}
\lambda^*\biggl[\int_\Omega \Big(\mid \nabla u_*\mid ^{p-2}&+\mid \nabla u_*\mid ^{q-2}\Big)\nabla u_* \cdot \nabla v~dx -
\lambda\Big(\int_\Omega a\mid  u_*\mid ^{q-2} u_*  v~dx+\int_{\partial\Omega} b \mid  u_*\mid ^{q-2} u_*  v~d\sigma\Big)\biggr]\\
&+\mu_1\biggl[p\int_\Omega \mid \nabla u_*\mid ^{p-2}\nabla u_* \cdot \nabla v~dx+q \int_\Omega\mid\nabla u_*\mid ^{q-2}\nabla u_* \cdot \nabla v~dx \\
&-q\lambda\Big(\int_\Omega a\mid  u_*\mid ^{q-2} u_*  v~dx+\int_{\partial\Omega} b \mid  u_*\mid ^{q-2} u_*  v~d\sigma\Big)\biggr]\\
&+\mu_2(q-1)\biggl[\int_\Omega a\mid  u_*\mid ^{q-2}   v~dx+\int_{\partial\Omega} b \mid  u_*\mid ^{q-2}   v~d\sigma\biggr]=0~~\forall v\in W.
\end{split}
\end{equation*}
Testing with $v\equiv 1$ in the above equation and taking into account the fact that $u_*\in \mathcal{N}_\lambda$ we find
\[
\mu_2(q-1)\biggl[(\int_\Omega a\mid  u_*\mid ^{q-2}  ~dx+\int_{\partial\Omega} b \mid  u_*\mid ^{q-2}  ~d\sigma\bigr]=0,
\]
which implies $\mu_2=0$ (since the coefficient of $\mu_2$ in the above equation $\neq 0$; see \eqref{eq:3.15} with $v=u_*$ and Remark~\ref{remark11}). 

Next, we test with $v=u_*$ and use \eqref{eq:3.16} with $v=u_*$ to obtain 
\[
\mu_1(p-q)\int_\Omega \mid \nabla u_*\mid ^{p}~dx=0,
\]
 which implies $\mu_1=0$. Therefore, $\lambda^*\neq 0$, hence
\[
\int_\Omega \Big(\mid \nabla u_*\mid ^{p-2}+\mid \nabla u_*\mid ^{q-2}\Big)\nabla u_* \cdot \nabla v~dx =\lambda\Big(\int_\Omega a\mid  u_*\mid ^{q-2} u_*  v~dx+\int_{\partial\Omega} b \mid  u_*\mid ^{q-2} u_*  v~d\sigma\Big),
\]
$\forall~v\in W$, i. e.  $\lambda$ is indeed an eigenvalue of problem \eqref{eq:1.1}.
This completes the proof of the theorem.
\end{proof}

\begin{rem}
Assume that $(H_{\Omega})$, $(H_{ab})$ are fulfilled and $q\ge 2$. We can show that, if in addition $1<p<q$, 
then $\lambda_1$ (defined in \eqref{eq:2.3}) is in fact the first positive eigenvalue of the eigenvalue problem 
\begin{equation}\label{eq:3.17}
\left\{\begin{array}{l}
-\Delta_q u=\lambda a(x) \mid u\mid ^{q-2}u\hspace*{1cm} \ \mbox{in} ~ \Omega,\\[1mm]
\mid \nabla u\mid ^{q-2}\frac{\partial u}{\partial\nu}=\lambda b(x) \mid u\mid ^{q-2}u ~~~~~~ \mbox{on} ~ \partial\Omega.
\end{array}\right.
\tag{3.17}
\end{equation}
As in the case of problem \eqref{eq:1.1}, $\lambda \in \mathbb{R}$ is called an eigenvalue of problem \eqref{eq:3.17} if there exists 
$u_\lambda \in W^{1,q}(\Omega) \setminus \{0\}$ such that
\begin{equation}\label{eq:3.18}
\int_\Omega \mid \nabla u_\lambda\mid ^{q-2}\nabla u_\lambda  \nabla v~dx =\lambda \Big(\int_\Omega a\mid  u_\lambda\mid ^{q-2} u_\lambda  v~dx+
\int_{\partial\Omega} b \mid  u_\mu\mid ^{q-2} u_\mu  v~d\sigma\Big)~~\forall v\in W^{1,q}(\Omega).
\tag{3.18}
\end{equation}
Obviously, $\lambda_0=0$ is an eigenvalue of \eqref{eq:3.17} and any other eigenvalue of this problem is positive (cf. \eqref{eq:3.18} with 
$v=u_\lambda$). For $q\geq 2,$ we can use Lemma~\ref{lema2} to show that the first positive eigenvalue of \eqref{eq:3.17} is given by 
\begin{equation}\label{eq:3.19}
\lambda_{1q}:=\underset{v\in\mathcal{C}_q\setminus\{0\}}{\inf }~\frac{\int_\Omega\mid\nabla v\mid^q~dx}{\int_{\Omega}a\mid v\mid^q~dx+
\int_{\partial\Omega}b\mid  v\mid^q~d\sigma}.
\tag{3.19}
\end{equation}
 First of all we see that there is no eigenvalue of \eqref{eq:3.17} 
in the interval $(0,\lambda_{1q})$. Assume the contrary, that there exists a $\lambda \in (0, \lambda_{1q})$ which is an eigenvalue and 
let $u_\lambda\in \mathcal{C}_q\setminus\{0\}$ be a correspunding eigenfunction. If we choose in \eqref{eq:3.18}
 $v=u_{\lambda}$ we get
\begin{equation}\label{eq:3.20}
\int_\Omega \mid \nabla u_\lambda\mid ^{q}~dx=\mu\Big(\int_\Omega a\mid  u_\lambda \mid ^{q}~dx+
\int_{\partial\Omega} b \mid  u_\lambda \mid ^{q} ~d\sigma\Big).
\tag{3.20}
\end{equation}
As $u_\lambda \in \mathcal{C}_q\setminus\{0\}$,  we have (see \eqref{eq:3.19})
\begin{equation*}
\lambda< \lambda_{1q}\leq \frac{\int_\Omega\mid\nabla u_\lambda\mid^q~dx}{\int_{\Omega}a\mid
 u_\lambda \mid^q~dx+\int_{\partial\Omega}b\mid u_\lambda\mid^q~d\sigma}=\lambda,
\end{equation*}
contradiction. Now, let us prove that $\lambda_{1q}$ is an eigenvalue of \eqref{eq:3.17}. We know from Lemma~\ref{lema1} that there exists $u^*\in 
\mathcal{C}_{1q}\setminus \{ 0\}$ such that 
$$
\lambda_{1q}=J(u^*)=\underset{v\in\mathcal{C}_{1q}}{\min }~J(v).
$$
We can prove that $J'(u^*)=0$. To this purpose we apply Lemma ~\ref{lema2} to problem \eqref{eq:2.2} with the constraints:
\begin{equation}\label{eq:3.21}
h_1(v)=\int_\Omega a\mid  v\mid ^{q} ~dx+\int_{\partial\Omega} b \mid  v\mid ^{q} ~d\sigma-1=0,
\tag{3.21}
\end{equation}
\begin{equation}\label{eq:3.22}
h_2(v)=\int_\Omega a\mid  v\mid ^{q-2} v  ~dx+\int_{\partial\Omega} b \mid  v\mid ^{q-2} v  ~d\sigma=0.
\tag{3.22}
\end{equation}
Choose $X=W^{1,q}(\Omega), \ Y=\mathbb{R}^2, \ D=X, \ f=J, \ g=(h_1, h_2), \ x_0=u^*$. One can show by arguments similar to those used before that 
$g'(u^*)$ is surjective, so all the requirements of Lemma \ref{lema2} are fulfilled. So there exist $\lambda^* \in \mathbb{R}$, 
$(\mu_1, \mu_2) \in \mathbb{R}^2$, 
not both equal to zero such that
\begin{equation}\label{eq:3.23}
\begin{split}
\lambda^*q\int_\Omega &\mid \nabla u^*\mid ^{q-2}\nabla u^*  \nabla v~dx+\mu_{1}q \Big(\int_\Omega a\mid  u^*\mid ^{q-2}  u^* v~dx
+\int_{\partial\Omega} b \mid  u^*\mid ^{q-2} u^*   v~d\sigma\Big)
 \\ &+\mu_2(q-1)\Big(\int_\Omega a\mid  u^*\mid ^{q-2}   v~dx+\int_{\partial\Omega} b \mid  u^*\mid ^{q-2}   v~d\sigma\Big)=0~~\forall v\in W^{1,q}(\Omega),
\end{split}
\tag{3.23}
\end{equation}
Testing with $v=1$ in \eqref{eq:3.23} and observing that $u^*$ satisfies \eqref{eq:3.22}, we deduce that $\mu_2=0.$ Finally, chosing $v=u^*$ in 
\eqref{eq:3.23} and noting that $u^*$ satisfies \eqref{eq:3.21} we find $\lambda^* \lambda_{1q}+\mu_1=0$, where $\mu_1 \neq 0$, 
$\lambda^* \neq 0$. Replacing $\mu_1=-\lambda^* \lambda_{1q},~\mu_2=0$ in \eqref{eq:3.23}, we get 
$$
\int_\Omega \mid \nabla u^*\mid ^{q-2}\nabla u^*  \nabla v~dx - \lambda_{1q}\Big(\int_\Omega a\mid  u^*\mid ^{q-2}  u^* v~dx
+\int_{\partial\Omega} b \mid  u^*\mid ^{q-2} u^*   v~d\sigma\Big) \ \ \forall v \in W^{1,q}(\Omega), 
$$
i. e., $(\lambda_{1q}, u^*)$ is an eigenpair of problem \eqref{eq:3.17}. 

Thus, if $q\ge 2$ and $1<p<q$ then $\lambda_1 = \lambda_{1q}$, so the eigenvalue set of problem \eqref{eq:1.1} is $\{ 0\}\cup (\lambda_{1q}, \infty)$, 
which is independent of $p$. If $2\le q<p$ then $\lambda_1\ge \lambda_{1q}$.

\end{rem}


\begin{thebibliography}{}

\bibitem{AM} J. Abreu, G. Madeira, Generalized eigenvalues of the $(p,2)$-Laplacian under a parametric boundary condition, ArXiv:1507.03299v2 [math.AP] 23 Mar 2016.

\bibitem{Br} H. Brezis, Functional Analysis, Sobolev Spaces and Partial Differential
Equations, Springer, 2011.

\bibitem{CF} E. Casas, L. A. Fern\'{a}ndez, A Green's formula for quasilinear elliptic operators, J. Math. Anal. Appl., 142, 62-73, 1989.

\bibitem{DMP} Z. Denkowski, S. Mig\'{o}rski, N. S. Papageorgiou, An Introduction to Nonlinear
Analysis: Theory, Springer, New York, 2003.

\bibitem{FMS} M. F\u{a}rc\u{a}\c{s}eanu, M. Mih\u{a}ilescu, D. Stancu-Dumitru, On the set of eigen-
values of some PDEs with homogeneous Neumann boundary condition, Nonlinear Analysis,  116, 19-25, 2015.


\bibitem{Le} A. L\^{e}, Eigenvalue problems for p-Laplacian, Nonlinear Analysis, 64,1057-1099, 2006.

\bibitem{GP}  L. Gasinski, N. S. Papageorgiou, Nonlinear Analysis, Chapman and Hall/CRC
Taylor and Francis Group, Boca Raton, 2005.

\bibitem{MMih} M. Mih\u{a}ilescu, An eigenvalue problem possesing a continuous family of eigenvalues plus an isolated eigenvale, Commun. Pure
Appl. Anal. 10, 701-708, 2011.

\bibitem{MM} M. Mih\u{a}ilescu, G. Moro\c{s}anu, Eigenvalues of $-\triangle_p-\triangle_q$ under Neumann
boundary condition, Canadian Math. Bull., 59(3), 606-616, 2016.

\bibitem{papa} N. S. Papageorgiou, S. Th. Kyritsi-Yiallourou, Handbook on Applied Analysis. Advances in Mechanics and Mathematics, 19, Springer,
New York, 2009.

\bibitem{St}  M. Struwe, Variational Methods: Applications to Nonlinear Partial Differential Equations and Hamiltonian Systems,
Springer, 1996.

\bibitem{SW} A. Szulkin, T. Weth, The Method of Nehary Manifold, Handbook of
Nonconvex Analysis and Applications, Int. Press, Somerville, MA, 597-632, 2010.


\end{thebibliography}
\end{document}